\documentclass{article}

\RequirePackage{amssymb,amsthm,amsmath}
\RequirePackage[colorlinks,citecolor=blue,urlcolor=blue]{hyperref}
\usepackage{bm}

\newtheorem{Theorem}{Theorem}
\newtheorem{Lemma}{Lemma}
\newtheorem{Corollary}{Corollary}

\newcommand{\cost}{\mathbf{cost}}
\newcommand{\area}{\mathbf{area}}

\newcommand{\Ints}{{\mathbb{Z}}}
\newcommand{\CC}{{\mathcal C}}

\newcommand{\GG}{{\mathcal G}}

\newcommand{\Npair}{N_{\mbox{\scriptsize{pair}}}}

 \newcommand{\Ex}{{\mathbb E}}
 \renewcommand{\Pr}{{\mathbb{P}}}
 \newcommand{\sfrac}[2]{{\textstyle\frac{#1}{#2}}}

\DeclareMathOperator{\Var}{var}

 \title{Route Lengths in Invariant Spatial Tree Networks}
 \author{David J. Aldous\thanks{Research supported by NSF Grant DMS-1504802.}\\
Department of Statistics\\
367 Evans Hall \#\  3860\\
U.C. Berkeley CA 94720\\
aldous@stat.berkeley.edu}

\begin{document}
\maketitle

\begin{abstract}
Is there a constant $r_0$ such that, 
in any invariant tree network linking rate-$1$ Poisson points in the plane, the mean within-network distance between points at Euclidean distance $r$ is infinite for  $r > r_0$?
We prove a slightly weaker result.
This is a continuum analog of a result of Benjamini et al (2001) on invariant spanning trees of the integer lattice.
\end{abstract}

\section{Introduction}
Parts of classical stochastic geometry \cite{stoyan}, 
for instance Delaunay triangulations on random points, 
implicitly concern random spatial networks but without direct motivation as real-world network models.
Substantial recent literature, surveyed in the 2018 monograph \cite{barthelemy}, 
concerns toy models of more specific types of real-world spatial network, studied in statistical physics style rather than
theorem-proof style.
Intermediate between those styles, 
and  envisioning examples such as inter-city road networks, 
one can model the city positions as a Poisson point process, 
and one can study the trade-off
between a network's cost (taken as network length) and its effectiveness at providing short routes  \cite{lando,kendall,shun}. 
It is often remarked that tree networks are {\em obviously} very ineffective at providing short routes,
and the purpose of this article is to give one formalization, as Theorem \ref{T1}.

As background we mention two results for {\em lattice} models.
Consider $m^2$ cities  
at the vertices of the $m \times m$ grid.  
Any connected network must have length $\Omega(m^2)$, and the mean route-length between 
two uniform random points must be $\Omega(m)$.
Observe that these orders of magnitude can be attained by a tree-network; from each vertex create a unit 
edge to a neighbor vertex nearer to a central root. 
This type of construction extends readily to the Poisson model.
But this apparent ``linearity of mean route lengths" is in some ways misleading, in that it depends on a finite network having a central region.
{\em Infinite} tree networks with a spatial stationarity property are different,
as shown by the 
following elegant result of Benjamini et al. \cite{peres} in the infinite lattice setting.  
Here {\em invariant} means the distribution of the network is invariant under the automorphisms of the lattice.\footnote{Theorem \ref{T:peres} was stated 
in \cite{peres} Theorem 14.3 for a particular model, but as noted in \cite{lyons} Exercise 4.48 it holds in the general 
automorphism-invariant case.}
\begin{Theorem}[\cite{peres}]
\label{T:peres}
For any invariant random spanning tree in the infinite 2-dimensional square lattice, the 
(within-tree) route length $D$ between lattice-adjacent vertices satisfies
\[ \Pr(D \ge i) \ge \sfrac{1}{8i},\  i \ge 1 . \] 
In particular, $\Ex D = \infty$.
\end{Theorem}
A relation between finite models and infinite invariant models is provided by  local weak convergence, discussed briefly in section \ref{sec:LWC}.

The proof of Theorem \ref{T:peres} exploits symmetries of the lattice which clearly 
are not directly applicable in the Poisson model.
So what is the analog of Theorem \ref{T:peres} in the rate-$1$ Poisson model on the plane?
Here {\em invariant} means the distribution of the network  is invariant under the Euclidean group.
We would like to consider
\begin{equation}
\rho(r) := \mbox{ mean route length between two Poisson points at distance } r .
\label{def:rho}
\end{equation}  
As noted in section \ref{sec:MST}, the MST (minimum spanning tree) provides a model
 in which $\rho(r) < \infty$ for small $r$.  
It seems natural to conjecture  that there exists a constant $r_0 <  \infty $ such that, 
for all invariant tree networks over the rate-$1$ Poisson process, 
$\rho(r) = \infty$ for a.a. $r \ge r_0$.
To avoid possible very artificial examples (see section \ref{sec:counter})
we actually prove a slightly weaker assertion, 
by considering instead the route-length $D_r$ between Poisson points at distance {\em at most} $r$.

To be precise about the meaning of {\em tree-network}, we allow Steiner points 
(junctions, envisaging road networks) as vertices in addition to the given Poisson points. 
And we take edges to be line segments between vertices.  
The {\em tree} property is that there are no circuits.

\begin{Theorem}
\label{T1}
There exist constants $r_0 <  \infty $ and $\beta > 0$ such that, in
every invariant tree-network connecting the points of a Poisson point process of rate $1$ in the infinite plane,
 for $r \ge r_0$ 
\[ \Pr(D_r > d) \ge \beta r/d, \ r \le d < \infty \]
and so $\Ex D_r = \infty$ for $r \ge r_0$.
\end{Theorem}
So this is a continuum analog of Theorem \ref{T:peres}.
The proof in section \ref{sec:proofs} relies on the fact that a finite tree has a {\em centroid} 
from which each branch contains less than half the vertices; 
the route between two  vertices in different branches must go via the centroid, 
so the route length is lower bounded by the sum of distances to the centroid.
Consider the partition of a very large square into a large number of large subsquares.
If there are a non-negligible number of subsquares in which points from more than one branch have non-negligible relative frequency,
then the point-pairs within such subsquares provide the desired long routes.
Otherwise almost all subsquares have almost all points from the same branch, but therefore
(and this is the key intricate technical issue, Lemma \ref{L:green}) 
there must be some number of pairs of adjacent subsquares for which these are {\em different} branches, and so (by the easy Lemma \ref{L:Red}) some overlapping square has a substantial proportion of its points from different branches, which as before provide the desired long routes.

Our proof is technically elementary, albeit rather intricate, using only very basic facts from percolation theory.
It seems quite likely that some shorter proof could be found, using some more sophisticated percolation theory.

Remarks on analogous questions for general networks are given in section \ref{sec:rem}.
Note also that, for Theorem \ref{T1} to be interesting in the sense of generality, one would like to know that there are many different ways to construct invariant tree-networks over Poisson points, and we discuss this  in section \ref{sec:LWC}.

\section{Proofs}
\label{sec:proofs}

\subsection{Technical lemmas}
Here we give two lemmas.
The first, which is elementary,  will enable reduction to a lattice percolation setting, and the second is the key technical ingredient we need in that setting.
To aid intuition we state these in terms of colorings, though with different  interpretations in the two lemmas, and it is {\bf not} the graph-theoretic {\em coloring} notion in which adjacent vertices must have different colors.

Fix a large integer $m$.
Call a configuration of points in general position in the continuum $m \times m$ square $S_1 = [0,m]^2$ {\em balanced} if,
in each of the ten sub-rectangles 
$[(i-1)m/5,im/5] \times [0,1], 1 \le i \le 5$ and 
$[0,1] \times  [(i-1)m/5,im/5] , 1 \le i \le 5$, 
the number of points is between $0.98 m^2/5$ and $1.02 m^2/5$.
Make the analogous definition for the adjacent square $S_2 = [m,2m] \times [0,m]$.

\begin{Lemma}
\label{L:Red}
Suppose $S_1$ and $S_2$ each contain a balanced configuration of points. 
Consider  a   $\{$blue, red$\}$ coloring of the points in $S_1 \cup S_2$, and suppose that neither \\
(a) $S_1$ and $S_2$ both contain less than $0.1 m^2$ blue points\\
nor (b) $S_1$ and $S_2$ both contain more than $0.88 m^2$ blue points\\
is true.
Then the number of blue-red  point pairs at distance at most $2^{1/2} m$ apart is at least $0.088m^4$.
\end{Lemma}
Note we are counting {\em all} such pairs, not asking for a matching where a point can be in only one pair.

\begin{proof}
First, if either $S_1$ or $S_2$ contains between $0.1 m^2$  and $0.88 m^2$ blue points, say $y$ blue points,
then (from the definition of {\em balanced}) there are at least 
$0.98 m^2 - y$ red points, and so at least $y(0.98 m^2 - y) \ge
0.1 m^2 \times 0.88 m^2$ blue-red pairs within that square. Such a pair is at most $2^{1/2}m$ apart.
The only remaining case is w.l.o.g where 
 $S_1$ contains less than $0.1 m^2$ blue points, and
 $S_2$ contains more than $0.88 m^2$ blue points.
In this case, consider the successive translated squares $[im/5,m+im/5] \times [0,m],  i = 0, 1, 2, \ldots, 5$.
At each step the number of blue points can increase by at most $1.02 m^2/5$,
so in at least one of the translated squares there are between $0.1 m^2$  and $0.88 m^2$ blue points, and the result follows as in the first case.
\end{proof}

\newpage

For our key technical lemma, fix a large integer $k$ and consider the $k \times k$ grid graph with vertices 
$G_k = \{0,1,,\ldots, k-1\} \times  \{0,1,,\ldots, k-1\}$.
\begin{Lemma}
\label{L:green}
Given an arbitrary subset $\xi^k$ of $G_k$, let $c(\xi^k)$ be the minimum, 
over all $\{$green-yellow$\}$ colorings of $G_k$ with at least  $k^2/4$ vertices of each color, 
of the number of green-yellow adjacent pairs where neither vertex is in $\xi^k$.
Then there exists $q>0$ such that, taking $\Xi^k$ 
to be the random subset in which each vertex is present independently with probability $q$,
\[ \Pr( c(\Xi^k) <  k/400) \to 0 \mbox{ as } k \to \infty  . \]
\end{Lemma}
As motivation, in the proof of Theorem \ref{T1} we will apply this where the vertices represent large squares and the two colors indicate a relatively large or relatively small number 
of points in a given tree-branch in the square.
The proof of Lemma \ref{L:green} is in essence just the classical {\em Peierls contour} method \cite{meester}, but applied in two different ways.

\setlength{\unitlength}{0.14in}
\begin{figure}

\begin{picture}(30,15)(-5,0)

\put(0,0){\circle{0.3}}
\put(1,0){\circle{0.3}}
\put(2,0){\circle*{0.3}}
\put(3,0){\circle*{0.3}}
\put(4,0){\circle*{0.3}}
\put(5,0){\circle*{0.3}}
\put(6,0){\circle{0.3}}
\put(7,0){\circle{0.3}}
\put(8,0){\circle{0.3}}
\put(9,0){\circle*{0.3}}
\put(10,0){\circle*{0.3}}
\put(11,0){\circle*{0.3}}
\put(0,1){\circle{0.3}}
\put(1,1){\circle*{0.3}}
\put(2,1){\circle*{0.3}}
\put(3,1){\circle*{0.3}}
\put(4,1){\circle{0.3}}
\put(5,1){\circle{0.3}}
\put(6,1){\circle{0.3}}
\put(7,1){\circle{0.3}}
\put(8,1){\circle{0.3}}
\put(9,1){\circle*{0.3}}
\put(10,1){\circle*{0.3}}
\put(11,1){\circle*{0.3}}
\put(0,2){\circle*{0.3}}
\put(1,2){\circle*{0.3}}
\put(2,2){\circle{0.3}}
\put(3,2){\circle*{0.3}}
\put(4,2){\circle{0.3}}
\put(5,2){\circle{0.3}}
\put(6,2){\circle{0.3}}
\put(7,2){\circle{0.3}}
\put(8,2){\circle{0.3}}
\put(9,2){\circle{0.3}}
\put(10,2){\circle{0.3}}
\put(11,2){\circle*{0.3}}
\put(0,3){\circle*{0.3}}
\put(1,3){\circle*{0.3}}
\put(2,3){\circle{0.3}}
\put(3,3){\circle{0.3}}
\put(4,3){\circle*{0.3}}
\put(5,3){\circle*{0.3}}
\put(6,3){\circle*{0.3}}
\put(7,3){\circle*{0.3}}
\put(8,3){\circle{0.3}}
\put(9,3){\circle{0.3}}
\put(10,3){\circle{0.3}}
\put(11,3){\circle*{0.3}}
\put(0,4){\circle*{0.3}}
\put(1,4){\circle*{0.3}}
\put(2,4){\circle{0.3}}
\put(3,4){\circle{0.3}}
\put(4,4){\circle*{0.3}}
\put(5,4){\circle*{0.3}}
\put(6,4){\circle*{0.3}}
\put(7,4){\circle*{0.3}}
\put(8,4){\circle*{0.3}}
\put(9,4){\circle*{0.3}}
\put(10,4){\circle*{0.3}}
\put(11,4){\circle{0.3}}
\put(0,5){\circle{0.3}}
\put(1,5){\circle{0.3}}
\put(2,5){\circle{0.3}}
\put(3,5){\circle*{0.3}}
\put(4,5){\circle*{0.3}}
\put(5,5){\circle*{0.3}}
\put(6,5){\circle{0.3}}
\put(7,5){\circle{0.3}}
\put(8,5){\circle{0.3}}
\put(9,5){\circle*{0.3}}
\put(10,5){\circle*{0.3}}
\put(11,5){\circle{0.3}}
\put(0,6){\circle{0.3}}
\put(1,6){\circle{0.3}}
\put(2,6){\circle{0.3}}
\put(3,6){\circle{0.3}}
\put(4,6){\circle*{0.3}}
\put(5,6){\circle*{0.3}}
\put(6,6){\circle{0.3}}
\put(7,6){\circle*{0.3}}
\put(8,6){\circle*{0.3}}
\put(9,6){\circle*{0.3}}
\put(10,6){\circle*{0.3}}
\put(11,6){\circle{0.3}}
\put(0,7){\circle{0.3}}
\put(1,7){\circle{0.3}}
\put(2,7){\circle{0.3}}
\put(3,7){\circle{0.3}}
\put(4,7){\circle*{0.3}}
\put(5,7){\circle*{0.3}}
\put(6,7){\circle*{0.3}}
\put(7,7){\circle*{0.3}}
\put(8,7){\circle*{0.3}}
\put(9,7){\circle*{0.3}}
\put(10,7){\circle*{0.3}}
\put(11,7){\circle{0.3}}
\put(0,8){\circle*{0.3}}
\put(1,8){\circle*{0.3}}
\put(2,8){\circle{0.3}}
\put(3,8){\circle{0.3}}
\put(4,8){\circle*{0.3}}
\put(5,8){\circle*{0.3}}
\put(6,8){\circle*{0.3}}
\put(7,8){\circle{0.3}}
\put(8,8){\circle{0.3}}
\put(9,8){\circle{0.3}}
\put(10,8){\circle{0.3}}
\put(11,8){\circle{0.3}}
\put(0,9){\circle*{0.3}}
\put(1,9){\circle*{0.3}}
\put(2,9){\circle*{0.3}}
\put(3,9){\circle{0.3}}
\put(4,9){\circle{0.3}}
\put(5,9){\circle{0.3}}
\put(6,9){\circle{0.3}}
\put(7,9){\circle{0.3}}
\put(8,9){\circle{0.3}}
\put(9,9){\circle{0.3}}
\put(10,9){\circle{0.3}}
\put(11,9){\circle*{0.3}}
\put(0,10){\circle*{0.3}}
\put(1,10){\circle*{0.3}}
\put(2,10){\circle{0.3}}
\put(3,10){\circle{0.3}}
\put(4,10){\circle{0.3}}
\put(5,10){\circle{0.3}}
\put(6,10){\circle{0.3}}
\put(7,10){\circle{0.3}}
\put(8,10){\circle*{0.3}}
\put(9,10){\circle{0.3}}
\put(10,10){\circle*{0.3}}
\put(11,10){\circle*{0.3}}
\put(0,11){\circle*{0.3}}
\put(1,11){\circle*{0.3}}
\put(2,11){\circle{0.3}}
\put(3,11){\circle{0.3}}
\put(4,11){\circle{0.3}}
\put(5,11){\circle{0.3}}
\put(6,11){\circle*{0.3}}
\put(7,11){\circle*{0.3}}
\put(8,11){\circle*{0.3}}
\put(9,11){\circle{0.3}}
\put(10,11){\circle*{0.3}}
\put(11,11){\circle*{0.3}}
\put(-0.5,1.5){\line(1,0){1}}
\put(0.5,1.5){\line(0,-1){1}}
\put(0.5,.5){\line(1,0){1}}
\put(1.5,0.5){\line(0,-1){1}}
\put(-0.5,4.5){\line(1,0){2}}
\put(1.5,4.5){\line(0,-1){3}}
\put(1.5,1.5){\line(1,0){1}}
\put(2.5,1.5){\line(0,1){1}}
\put(2.5,2.5){\line(1,0){1}}
\put(3.5,2.5){\line(0,-1){2}}
\put(3.5,0.5){\line(1,0){2}}
\put(5.5,0.5){\line(0,-1){1}}
\put(8.5,-0.5){\line(0,1){2}}
\put(8.5,1.5){\line(1,0){2}}
\put(10.5,1.5){\line(0,1){2}}
\put(10.5,3.5){\line(1,0){1}}
\put(3.5,2.5){\line(1,0){4}}
\put(7.5,2.5){\line(0,1){1}}
\put(7.5,3.5){\line(1,0){3}}
\put(10.5,3.5){\line(0,1){4}}
\put(10.5,7.5){\line(-1,0){4}}
\put(6.5,7.5){\line(0,1){1}}
\put(6.5,8.5){\line(-1,0){3}}
\put(3.5,8.5){\line(0,-1){3}}
\put(3.5,5.5){\line(-1,0){1}}
\put(2.5,5.5){\line(0,-1){1}}
\put(2.5,4.5){\line(1,0){1}}
\put(3.5,4.5){\line(0,-1){2}}
\put(5.5,4.5){\line(1,0){3}}
\put(8.5,4.5){\line(0,1){1}}
\put(8.5,5.5){\line(-1,0){2}}
\put(6.5,5.5){\line(0,1){1}}
\put(6.5,6.5){\line(-1,0){1}}
\put(5.5,6.5){\line(0,-1){2}}
\put(9.5,11.5){\line(0,-1){2}}
\put(9.5,9.5){\line(1,0){1}}
\put(10.5,9.5){\line(0,- 1){1}}
\put(10.5,8.5){\line(1,0){1}}
\put(-0.5,7.5){\line(1,0){2}}
\put(1.5,7.5){\line(0,1){1}}
\put(1.5,8.5){\line(1,0){1}}
\put(2.5,8.5){\line(0,1){1}}
\put(2.5,9.5){\line(-1,0){1}}
\put(1.5,9.5){\line(0,1){2}}
\put(5.5,11.5){\line(0,-1){1}}
\put(5.5,10.5){\line(1,0){2}}
\put(7.5,10.5){\line(0,-1){1}}
\put(7.5,9.5){\line(1,0){1}}
\put(8.5,9.5){\line(0,1){2}}
\put(19,0){\circle*{0.3}}
\put(20,0){\circle*{0.3}}
\put(21,0){\circle*{0.3}}
\put(22,0){\circle{0.3}}
\put(24,0){\circle{0.3}}
\put(25,0){\circle*{0.3}}

\put(17,1){\circle*{0.3}}
\put(18,1){\circle*{0.3}}
\put(19,1){\circle*{0.3}}
\put(20,1){\circle{0.3}}
\put(21,1){\circle{0.3}}
\put(22,1){\circle{0.3}}
\put(24,1){\circle{0.3}}
\put(25,1){\circle*{0.3}}
\put(26,1){\circle*{0.3}}
\put(27,1){\circle*{0.3}}

\put(17,2){\circle*{0.3}}
\put(18,2){\circle{0.3}}
\put(19,2){\circle*{0.3}}
\put(20,2){\circle{0.3}}
\put(21,2){\circle{0.3}}
\put(22,2){\circle{0.3}}
\put(23,2){\circle{0.3}}
\put(24,2){\circle{0.3}}
\put(25,2){\circle{0.3}}
\put(26,2){\circle{0.3}}
\put(27,2){\circle*{0.3}}

\put(17,3){\circle*{0.3}}
\put(18,3){\circle{0.3}}
\put(19,3){\circle{0.3}}
\put(20,3){\circle*{0.3}}
\put(21,3){\circle*{0.3}}
\put(22,3){\circle*{0.3}}
\put(23,3){\circle*{0.3}}
\put(24,3){\circle{0.3}}
\put(25,3){\circle{0.3}}
\put(26,3){\circle{0.3}}
\put(27,3){\circle*{0.3}}

\put(16,4){\circle*{0.3}}
\put(17,4){\circle*{0.3}}
\put(18,4){\circle{0.3}}
\put(19,4){\circle{0.3}}
\put(20,4){\circle*{0.3}}
\put(23,4){\circle*{0.3}}
\put(24,4){\circle*{0.3}}
\put(25,4){\circle*{0.3}}
\put(26,4){\circle*{0.3}}
\put(27,4){\circle{0.3}}

\put(16,5){\circle{0.3}}
\put(17,5){\circle{0.3}}
\put(18,5){\circle{0.3}}
\put(19,5){\circle*{0.3}}
\put(20,5){\circle*{0.3}}
\put(26,5){\circle*{0.3}}
\put(27,5){\circle{0.3}}

\put(18,6){\circle{0.3}}
\put(19,6){\circle{0.3}}
\put(20,6){\circle*{0.3}}
\put(26,6){\circle*{0.3}}
\put(27,6){\circle{0.3}}

\put(16,7){\circle{0.3}}
\put(17,7){\circle{0.3}}
\put(18,7){\circle{0.3}}
\put(19,7){\circle{0.3}}
\put(20,7){\circle*{0.3}}
\put(22,7){\circle*{0.3}}
\put(23,7){\circle*{0.3}}
\put(24,7){\circle*{0.3}}
\put(25,7){\circle*{0.3}}
\put(26,7){\circle*{0.3}}
\put(27,7){\circle{0.3}}

\put(16,8){\circle*{0.3}}
\put(17,8){\circle*{0.3}}
\put(18,8){\circle{0.3}}
\put(19,8){\circle{0.3}}
\put(20,8){\circle*{0.3}}
\put(21,8){\circle*{0.3}}
\put(22,8){\circle*{0.3}}
\put(23,8){\circle{0.3}}
\put(24,8){\circle{0.3}}
\put(25,8){\circle{0.3}}
\put(26,8){\circle{0.3}}
\put(27,8){\circle{0.3}}

\put(17,9){\circle*{0.3}}
\put(18,9){\circle*{0.3}}
\put(19,9){\circle{0.3}}
\put(20,9){\circle{0.3}}
\put(21,9){\circle{0.3}}
\put(22,9){\circle{0.3}}
\put(23,9){\circle{0.3}}
\put(24,9){\circle{0.3}}
\put(25,9){\circle{0.3}}
\put(26,9){\circle{0.3}}
\put(27,9){\circle*{0.3}}

\put(17,10){\circle*{0.3}}
\put(18,10){\circle{0.3}}
\put(19,10){\circle{0.3}}
\put(21,10){\circle{0.3}}
\put(22,10){\circle{0.3}}
\put(23,10){\circle{0.3}}
\put(24,10){\circle*{0.3}}
\put(25,10){\circle{0.3}}
\put(26,10){\circle*{0.3}}
\put(27,10){\circle*{0.3}}

\put(17,11){\circle*{0.3}}
\put(18,11){\circle{0.3}}

\put(21,11){\circle{0.3}}
\put(22,11){\circle*{0.3}}
\put(23,11){\circle*{0.3}}
\put(24,11){\circle*{0.3}}
\put(25,11){\circle{0.3}}
\put(26,11){\circle*{0.3}}
\put(15.5,4.5){\line(1,0){2}}
\put(17.5,4.5){\line(0,-1){3}}
\put(17.5,1.5){\line(1,0){1}}
\put(18.5,1.5){\line(0,1){1}}
\put(18.5,2.5){\line(1,0){1}}
\put(19.5,2.5){\line(0,-1){2}}
\put(19.5,0.5){\line(1,0){2}}
\put(21.5,0.5){\line(0,-1){1}}
\put(24.5,-0.5){\line(0,1){2}}
\put(24.5,1.5){\line(1,0){2}}
\put(26.5,1.5){\line(0,1){2}}
\put(26.5,3.5){\line(1,0){1}}
\put(19.5,2.5){\line(1,0){4}}
\put(23.5,2.5){\line(0,1){1}}
\put(23.5,3.5){\line(1,0){3}}
\put(26.5,3.5){\line(0,1){4}}
\put(26.5,7.5){\line(-1,0){4}}
\put(22.5,7.5){\line(0,1){1}}
\put(22.5,8.5){\line(-1,0){3}}
\put(19.5,8.5){\line(0,-1){3}}
\put(19.5,5.5){\line(-1,0){1}}
\put(18.5,5.5){\line(0,-1){1}}
\put(18.5,4.5){\line(1,0){1}}
\put(19.5,4.5){\line(0,-1){2}}
\put(25.5,11.5){\line(0,-1){2}}
\put(25.5,9.5){\line(1,0){1}}
\put(26.5,9.5){\line(0,- 1){1}}
\put(26.5,8.5){\line(1,0){1}}
\put(15.5,7.5){\line(1,0){2}}
\put(17.5,7.5){\line(0,1){1}}
\put(17.5,8.5){\line(1,0){1}}
\put(18.5,8.5){\line(0,1){1}}
\put(18.5,9.5){\line(-1,0){1}}
\put(17.5,9.5){\line(0,1){2}}
\put(21.5,11.5){\line(0,-1){1}}
\put(21.5,10.5){\line(1,0){2}}
\put(23.5,10.5){\line(0,-1){1}}
\put(23.5,9.5){\line(1,0){1}}
\put(24.5,9.5){\line(0,1){2}}
\end{picture}
\caption{Maximal circuits and paths in the finite grid.}
\label{Fig:1}
\end{figure}
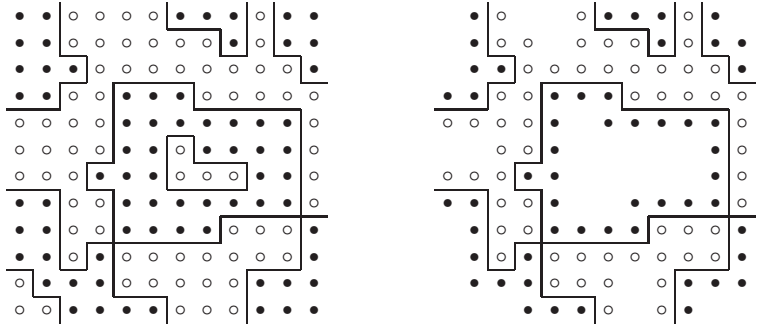

\begin{proof}
To recall basic percolation theory, in any coloring a green-yellow adjacent pair specifies an edge in a dual graph, 
and these edges form the boundaries of colored components.  
More precisely, as illustrated in Figure \ref{Fig:1}  (left), the set of such edges is a disjoint union of\\
(i) self-avoiding circuits within the $k \times k$ grid\\
(ii) self-avoiding paths starting and ending on the external dual boundary.\\
We write path* for ``path or circuit".
Fix $\ell > 4$ and consider a self-avoiding path* (in the dual graph)  $\pi$ of length $\ell$ in $G_k$.
Each edge separates some pair of vertices in $G_k$.
We can find a set $S_\pi$ of $\lfloor \ell/3\rfloor$ {\em disjoint} adjacent vertex pairs separated by some edge within $\pi$.
Consider the event $A_\pi$ that at most $\ell/20$ pairs within $S_\pi$ have neither end-vertex in $\Xi^k$.
This event has probability
\[ \Pr(A_\pi) = \Pr(\mathrm{Bin}(\lfloor \ell/3 \rfloor, (1-q)^2) \le \ell/20) . \]
The number of length-$\ell$  self-avoiding paths* $\pi$ is at most $4 k^2 3^{\ell - 1}$.
So the expected number of events $A_\pi$ that occur is at most
\begin{equation}
4 k^2 3^{\ell - 1} \times \Pr(\mathrm{Bin}(\lfloor \ell/3 \rfloor, (1-q)^2) \le \ell/20) .
\label{4k2}
\end{equation}
Setting $\ell(k) \sim \log k$, standard Binomial tail bounds imply that, for sufficiently small $q$, 
the quantity (\ref{4k2}) goes to $0$ as $k \to \infty$.
So we may assume
\begin{quote}
(*) For every self-avoiding path* $\pi$ of length $\ell(k)$ in the dual graph of $G_k$, there exist at least $\ell(k)/20$ 
disjoint adjacent vertex pairs separated by some edge within $\pi$ and with neither vertex in $\Xi^k$.
\end{quote}
Note this is a property of $\Xi^k$, not involving any coloring.

Now consider a green-yellow coloring  of $G_k$ with at least  $k^2/4$ vertices of each color, 
By an elementary argument, the length of the boundary within $G_k$ between colored regions, that is the sum of lengths 
of the paths* at (i,ii), is at least $k/2$.
Split that sum as $S_{long} + S_{short}$ according as the path* lengths are longer or shorter than $\ell(k)$.
If $S_{long} > k/10$ then,  by splitting these paths* into disjoint segments of length $\ell(k)$ as needed, 
property (*) easily implies existence of $k/400$ green-yellow adjacent pairs where neither vertex is in $\Xi^k$.

So it is enough to consider only colorings in which 
\begin{equation}
S_{long } \le k/10 \mbox{ and } S_{short} \ge k/2 - k/10 = 2k/5.
\label{long-short}
\end{equation}
Fix such a coloring, and consider the associated paths and circuits, as in Figure \ref{Fig:1}. 
Note that  a circuit (i) splits $G_k$ into an exterior and an {\em interior} region.
Also a  path (ii), which by (\ref{long-short}) has length $\le k/10$,  splits $G_k$ into a well-defined larger and a smaller region, where 
(somewhat confusingly) we designate the smaller region meeting  the boundary of $G_k$ as the {\em interior} of the path. 
It easily follows from $S_{long } \le k/10$ that at most $k^2/200$ vertices are inside long paths.
Also, at most $4 k \log k$ vertices can be inside short paths, so for large $k$
\begin{equation}
\mbox{ at most $k^2/100$ vertices are inside paths.}
\label{path-int}
\end{equation}
Here {\em inside} means ``in the interior of".

A {\em maximal} circuit is one that is not contained inside another circuit or path, 
and a {\em maximal} path  is one that is not  inside  another path.
Figure \ref{Fig:1}  (right) shows the 5 maximal paths and the 1 maximal circuit in that example.
Note that, by definition, there is a single-color path\footnote{In this specific context the path may include diagonals.} immediately inside and a single-opposite-color path immediately outside each maximal path or circuit.
Moreover the colors of these immediately-inside paths are the same (say $\bullet$) for each component, because a path in $G$ between 
a vertex in each component must cross component boundaries an even number of times. 
Every vertex of color $\bullet$ at  distance at least $\log k$ from the sides of $G_k$ is either inside some path, or inside some circuit and therefore inside some maximal circuit.
By hypothesis there are at least $k^2/4$ vertices of each color, so using (\ref{path-int}) and considering maximal circuits we have shown that 
property (\ref{long-short}) implies that for large $k$
\[ \mbox{there exist circuits, each of length less than $\log k$, with disjoint interiors } \]
\begin{equation}
\mbox{ and containing a total of at least $k^2/5$ vertices.  }
\label{ecb2}
\end{equation}

So it is enough to consider only colorings with property (\ref{ecb2}).
To analyze this case we need to set up some notation. 
Write $\Xi^\infty$ for the random subset of the {\em infinite} square lattice $\Ints^2$ in which each vertex is present independently with probability $q$.
Define the {\em cost} of a dual circuit $C$ in $\Ints^2$ to be the number of edges for which neither adjacent vertex is in $\Xi^\infty$, and similarly for dual circuits in $G_k$ and $\Xi^k$.
Consider the event
\[ A_\infty^q :=  \mbox{some circuit in $\Ints^2$ around the origin has zero cost} . \]
By a simpler use of the  Peierls contour method used for (\ref{4k2}),
$
\Pr(A_\infty^q) \to 0 \mbox{ as } q \downarrow 0. 
$
So we can fix $q$ sufficiently small that
\begin{equation}
\Pr(A_\infty^q) \le \sfrac{1}{20} .
\label{A12}
\end{equation}
Consider a coloring of $G_k$ (depending on $\Xi^k$) satisfying (\ref{ecb2}): 
to complete the proof of Lemma \ref{L:green} it will suffice to show that 
\[ N_k:= \mbox{ number of green-yellow adjacent pairs in $G_k$ with neither vertex in $\Xi^k$} \]
satisfies
\begin{equation}
 \Pr(N_k < k/400) \to 0 \mbox{ as } k \to \infty  . 
 \label{Nk20}
 \end{equation}
Write $\CC_k$ for the set of circuits guaranteed by (\ref{ecb2}), and write $\GG_k$ for the union of their interior vertices.
For $v \in \GG_k$ write $C_k(v)$ for the circuit in $\CC_k$ containing $v$ and $\area (C_k(v))$ for its area ( = number of interior vertices).
Now
\begin{eqnarray*}
N_k& \geq &\sum_{C_k \in \CC_k} \cost(C_k) \\
&=& \sum_{v \in \GG_k} \frac{1}{\area(C_k(v))} \cost(C_k(v)) .
\end{eqnarray*}
Now, taking $\Xi^k$ as the restriction of $\Xi^\infty$, we have
  $\cost(C_k(v)) \ge 1_{A_k^c(v)}$
where $A_k(v)$ is the event that some circuit around $v$ in $\Xi^\infty$ with length $\le \log k$ has zero cost.
Note by (\ref{A12})
\begin{equation}
 \Pr(A_k(v))  \le \Pr(A_\infty^q) \le \sfrac{1}{20} 
 \label{AkA}
 \end{equation}
and write
\[ N_k \ge  \sum_{v \in \GG_k} \frac{1}{\area(C_k(v))} \ 1_{A_k^c(v)} .\]
Each $\area(C_k(v))$ is at most $\log^2 k$, so
\begin{equation}
 N_k \ge \frac{1}{\log^2 k} \left( |\GG_k| - \sum_{v \in \GG_k} \ 1_{A_k(v)}  \right) \ge  \frac{1}{\log^2 k} \left( |\GG_k| - \sum_{v \in G_k} \ 1_{A_k(v)}  \right)     .
 \label{Nk=}
 \end{equation}
 By (\ref{AkA})
 \[
 \Ex  \left( \sum_{v \in G_k} \ 1_{A_k(v)}  \right) \le \sfrac{k^2}{20} 
 \]
If $v_1$ and $v_2$ are more than $\log k$ apart, the events $A_k(v_1)$ and $A_k(v_2)$ are independent, so 
\[ \Var \left(  \sum_{v \in G_k} \ 1_{A_k(v)}  \right) = O( k^2 \log^2 k )
\]
and then Chebyshev's inequality gives
\begin{equation}
 \Pr \left(  \sum_{v \in G_k} \ 1_{A_k(v)}   > k^2/10 \right) \to 0 \mbox{ as } k \to \infty . 
 \label{Cheb}
 \end{equation}
By (\ref{ecb2}) we have
$|\GG_k| \ge k^2/5$, and 
combining with (\ref{Nk=}) we find
\[ \Pr(N_k < k^2/(10 \log^2 k))  \to 0 \mbox{ as } k \to \infty \]
which is stronger than the desired bound (\ref{Nk20}).

To check the logic of this argument, note that the event in (\ref{Cheb}) involves only $\Xi^\infty$.
The other inequalities are deterministic, and show that, outside event (\ref{Cheb}), for every coloring satisfying (\ref{ecb2})  
and for large $k$, we have $N_k \ge k^2/(10 \log^2 k)$.

\end{proof}

We actually need the following  modification of Lemma \ref{L:green}, to say that the same result holds if we insist that
we count only pairs outside an arbitrary  subsquare of side $0.001k$. 
\begin{Corollary}
\label{C:1}
Let  $\Xi^k$ be the random subset of $G_k$ in which each vertex is present independently with probability $q$.
Let $\Box_k$ be a subsquare of $G_k$ of side asymptotic to $0.001k$, dependent on $\Xi^k$.
Let $c^\prime(\Xi^k)$ be the minimum, 
over all $\{$green-yellow$\}$ colorings of $G_k$ with at least  $k^2/4$ vertices of each color, 
of the number of green-yellow adjacent pairs where neither vertex is in $\Xi^k$ or in $\Box_k$  .
Then there exist $q>0$ and $\alpha > 0$ such that 
\[ \Pr( c^\prime(\Xi^k) < \alpha k) \to 0 \mbox{ as } k \to \infty  . \]
\end{Corollary}

\noindent
{\em Outline proof.}
Re-color the vertices in the small subsquare to become all the same color, and apply Lemma \ref{L:green} to the new configuration.
We omit details.

\subsection{Proof of Theorem \ref{T1}}
Take large integers $k$ and $m$, and set $n = km$.
Consider the  $n \times n$ square $[0,n]^2$ in the plane.
 Write $\Sigma_{m,k}$ for the index set of the natural partition of the square $[0,n]^2$ into $k^2$ 
 subsquares $\sigma$ of side $m$ -- call these the {\em natural} subsquares.
  The set $\Sigma_{m,k}$ is isomorphic to the $k \times k$ vertex grid $G_k$.
  In particular a subsquare $\Box$ of $G_k$, say with $s \times s$ vertices, 
  corresponds to a subsquare $\Box \hspace{-0.1in} + $ of the square $[0,n]^2$, with side $sm$, consisting of $s^2$ natural subsquares.
 
 Write $q_m$ for the probability that a realization of a rate-$1$ Poisson point process on a $m \times m$ square
 is not {\em balanced}, in the sense of Lemma \ref{L:Red}.  Clearly
 \begin{equation}
 q_m \to 0 \mbox{ as } m \to \infty .
 \label{qm}
 \end{equation}
 Given a rate-$1$ Poisson point process, the collection of not-balanced subsquares 
 can be identified with 
  the random subset $\Xi^k$ in Corollary \ref{C:1} 
  with $q = q_m$.
  In the bounds below we assume that $m$ and $k$ are sufficiently large, independently, 
 
Now consider a rate-$1$ Poisson point process on the whole plane and a tree-network connecting them. 
Write $N$ for the number of Poisson points in the square $[0,n]^2$.
Consider the subtree spanned by all the Poisson points in that square.
This subtree will typically extend outside the square. 
But there will exist a centroid, in the sense of a vertex $v_*$  (maybe a Steiner point, and maybe outside the square) 
such that, writing $B_1, B_2, \ldots$ for the sets of Poisson points within the square that are in the different branches from $v_*$,
the largest such set has size at most $N/2$.
It is then always possible to merge (if necessary) these sets into a bipartition 
$\{B,B^c\}$ of the points in the square such that $N/3 \le |B| \le N/2$.
 The key observation is that the path from any $v \in B$ to any $v^\prime \in B^c$ must go via $v_*$.  
 We will use this to prove the following key result, from which Theorem \ref{T1} will follow quite easily.
 \begin{Lemma}
\label{L:7}
There exists $\beta_0 > 0$ such that, with probability $\to 1$ as $m, k \to \infty$,
there are at least $\beta_0 m^4k$ pairs of points from $B$ and $B^c$ within straight-line distance $2^{1/2}m$
but whose route-length is at least $0.001mk$. 
\end{Lemma}
 \begin{proof}
  Color red the points in $B$, and color blue the points in $B^c$. 
 The ``probability" parts of the argument are the following easy consequences of the law of large numbers.
 Outside an event of probability $\to 0$ as $m, k \to \infty$:
 \begin{eqnarray}
& \mbox{ the total number of blue points and the total number of red points  are   $\ge0.33 m^2k^2$;} &\hspace*{0.1in} \label{tt1} \\
& \mbox{  the total number of points not in balanced natural subsquares  is at most $m^2k^2 \psi(m)$, }&  \nonumber \\
& \mbox{where $\psi(m) \downarrow 0$ as $m \to \infty$.} & \label{tt2}
\end{eqnarray}
The remainder of the argument is deterministic, and the precise numerical constants are not important.

 Write $\Box$ for a subsquare of $G_k$ with $0.001k \times 0.001k$ vertices, and 
 write $\Box \hspace{-0.1in} + $ for the corrresponding  square 
 of side $0.001n$ within $[0,n]^2$.
The essential issue is to find a  lower bound for
  \[
\Npair := \min_{\Box \hspace{-0.09in} + } \Npair(\Box \hspace{-0.1in} + ), \quad \mbox{ where }  \]
\vspace*{-0.15in}
 \[
\Npair(\Box \hspace{-0.1in} + )  := \mbox{ number of blue-red  point pairs outside $\Box \hspace{-0.1in} + $ and at distance at most $2^{1/2} m$ apart. }
\]
Write $b_\sigma$ for the number of blue points in the natural subsquare $\sigma$.  
Recall that a balanced natural subsquare must have between $0.98m^2$ and $1.02m^2$ points -- call this the {\em size condition}.
Amongst balanced natural subsquares, $\sigma$, consider
\begin{itemize}
\item the number $S$  for which  $0.09m^2 \le b_\sigma \le 0.89m^2$, 
\item the number $S^<$  for which $b_\sigma < 0.09m^2$,
\item the number $S^>$  for which $b_\sigma > 0.89m^2$. 
\end{itemize}
If  a balanced natural subsquare $\sigma$ has $0.09m^2 \le b_\sigma \le 0.89m^2$ then, by the size condition, there are at least $0.98m^2 - b_\sigma$ red points,
and so at least
$0.08m^4$ blue-red pairs in $\sigma$.
So we immediately have
\[
 \Npair \ge 0.08m^4S^* , \mbox{ where }
S^*:= \min_\Box S(\Box) \mbox{ and }  \]
 \[
S(\Box)  := \mbox{ number of balanced natural subsquares outside $\Box$ for which 
$0.09m^2 \le b_\sigma \le 0.89m^2$.}
\]
where (as above) $\Box$ is a subsquare of $G_k$ with $0.001k \times 0.001k$ vertices.
A given subsquare $\Box$ cannot intersect more than $0.000001k^2$ natural subsquares, and so 
\begin{equation}
 \Npair \ge 0.08m^4 (S -0.000001k^2)  .
 \label{pair1}
 \end{equation}
If $S$ is indeed of order $k^2$ then this inequality gives all we require (see (\ref{pair2})  below) but the key issue is to analyze the case where $S$ is small. 
We can lower bound the total number $N_{blue}$ of blue points by (\ref{tt1}) and upper bound it by  (\ref{tt2}) and the definitions of
$(S,  S^<, S^>)$: this gives
\[ 0.33m^2k^2 \le 0.09m^2 S^<  + 0.89 m^2 S + 1.02 m^2 S^> + m^2k^2 \psi(m) .    \]
Using $S^< + S + S^> \le k^2$ to eliminate the $S^<$ term, this rearranges to
\[ 0.24m^2k^2 \le 0.8m^2 S + 0.93 m^2 S^>     +    m^2k^2 \psi(m) .  \]
With the corresponding inequality arising from counting red points, we obtain that for $m$ sufficiently large
\begin{equation}
\mbox{if $S \le 0.005k^2$ then $\min(S^<, S^>) \ge k^2/4 $.}
\label{L3old}
\end{equation}

Color a natural subsquare $\sigma$ yellow if $b_\sigma < 0.1m^2$, or green otherwise.
In the case $\min(S^<, S^>) \ge k^2/4 $ we can
apply Corollary  \ref{C:1} provided $m$ is sufficiently large (recall (\ref{qm})),  to conclude that
(outside an event of probability $\to 0$ as $k \to \infty$) 
the number of adjacent balanced green-yellow pairs is at least $\alpha k$,
and these can be taken to avoid any choice of $\Box$ corresponding to a choice of $\Box \hspace{-0.1in} + $.
A given subsquare can be in at most $4$ such green-yellow pairs, so we can find $\alpha k/4$  disjoint pairs.
By Lemma \ref{L:Red}, within each such pair there exist at least $0.088m^4$ blue-red pairs
and so in this case
\begin{equation}
 \Npair \ge 0.088m^4 \times \alpha k/4 := \beta_0 m^4k 
 \label{pair2}
 \end{equation}
 for some constant $\beta$.   In the opposite case, that is by (\ref{L3old}) if $S > 0.005k^2$, 
 inequality  (\ref{pair1}) gives an essentially larger bound, so we may assume (\ref{pair2}).

Apply the bound (\ref{pair2}) where the given subsquare $\Box \hspace{-0.1in} + $ is a square of side $0.001mk$ centered near the tree centroid: 
we obtain the conclusion of Lemma \ref{L:7}.
\end{proof}

\paragraph{Completing the proof of Theorem \ref{T1}.}
Take expectation in Lemma \ref{L:7} and re-write in terms of 
\[ r:= 2^{1/2}m, \ \ d:= 0.001mk = 0.001 n \]
as follows. 

{\em (*) There are constants $\beta_1 > 0$ and $r_0, \rho_0 < \infty$ such that, for $r \ge r_0$ and $d/r \ge \rho_0$, 
and for any invariant tree model,
the mean number of pairs of Poisson points within $[0,n]^2$ at distance $\le r$ apart and with route-length $ \ge d$ is at least
$\beta_1 d r^3$.}

Let $\chi(r,d)$ be the probability, in a given model of an invariant tree-network over the Poisson points,
that between two typical Poisson points at distance $\le r$
 the route-length is $\ge d$. 
 The mean total number of pairs within $[0,n]^2$ at distance $\le r$ apart
  is bounded above by 
 $\frac{1}{2} n^2 \pi r^2$.
 So the mean number of such pairs with route-length  $\ge d$ is bounded above by 
  $\frac{1}{2} n^2 \pi r^2 \chi(r,d)$.
  This holds in particular when $d = 0.001n$, and now comparing 
  the upper and lower bound we find
    \begin{equation}
   \chi(r,d)   \ge \beta_2 \sfrac{r}{d} ; \quad r \ge r_0, d/r \ge \rho_0
  \label{chird}
  \end{equation}
  for a constant $\beta_2$.
  In the notation of Theorem \ref{T1} we have $\chi(r,d) = \Pr(D_r \ge d)$,
  and this inequality is equivalent to 
the form stated in Theorem \ref{T1}.

\section{Remarks}

\subsection{The Euclidean MST}
\label{sec:MST}
Consider the {\em random geometric graph} $\GG(r_0)$ whose  vertices form the rate-$1$ Poisson point process and whose edges link all pairs of points at Euclidean distance at most $r_0$.  
Write $N(v,r_0)$  for the number of vertices  in the component $C(v,r_0)$ of  $\GG(r_0)$   containing a typical vertex $v$. 
It is well known \cite{meester} that
\begin{equation}
\mbox{ for sufficiently small $r_0$, all moments of $N(v,r_0)$  are finite.}
\label{rsmall}
\end{equation}
Consider now the Euclidean MST over the Poisson process.
The restriction of the MST to $C(v,r_0)$ is a spanning tree within $C(v,r_0)$.
So the route length in the MST between $v$ and another vertex $v^\prime$ at distance $\le r_0$ is at most $r_0 N(v,r_0)$.
It easily follows (via a size-biasing argument) that the mean distance function $\rho(r)$ at (\ref{def:rho})  is such that $\rho(r_0)$ is finite when (\ref{rsmall}) holds.

\subsection{Constructing invariant tree-networks}
\label{sec:LWC}
Some examples are given in \cite{holroyd}; here is our general discussion.
Take an arbitrary tree-network linking $m^2$ independent uniform random vertices in the continuum square $[0,m]^2$,
and write $\ell_m$ for the expectation of the average (over vertices) length of the edge from the vertex toward the centroid.  
Randomly re-center, that is translate the plane as $(x,y) \to (x- U,y-V)$ for $(U,V)$ uniform on $[0,m]^2$, and then apply a uniform random rotation.
A sequence of such networks with $\ell_m$ bounded as $m \to \infty$ is tight in the natural ``local weak convergence" topology, and any subsequential weak limit network has invariant distribution.
This very general construction suggests that the class of invariant tree-networks should be very rich.
But there are two issues. 

In general the weak limit structure is guaranteed to be a forest with infinite tree-components, but is not guaranteed to be a single tree.
The planar MST limit is known to be a tree \cite{alexander} but the proof heavily exploits its explicit structure; 
there seem to be no useful general methods for proving that a construction via local weak convergence gives a limit tree.
To illustrate a more algorithmic construction, consider ``Poisson rain" on the plane -- rate $1$ per unit ares per unit time over time  $0<t \le 1$.
The construction rule
``each arriving point is a child of the nearest existing point" 
gives a genealogical tree studied in  \cite{partitions}.
Representing the parent-child relation by drawing a line segment, the network is not a tree because such lines may cross, but instead one can 
draw just the part of the segment from the child to the existing network, and the
analysis in   \cite{partitions} implies this will be a tree. 
Presumably other rules for connecting arriving points to the existing network within this Poisson rain framework will also yield invariant trees.

The second issue is illustrated by the notion of minimal (shortest length) Steiner tree.  
In the finite random setting this is a.s. unique.
The local weak convergence scheme produces  limit random forests attaining the minimum length-per-unit area possible, but 
(even if one could prove that limits are trees) it is not clear how to prove there is an a.s. {\em unique} limit tree attaining that minimum.

More abstractly, infinite trees arising as local weak limits are {\em unimodular}: general theory for unimodular trees at the graph-theoretic level 
is given in 
\cite{unimodular,baccelli} but is not specifically adapted to the spatial setting.

\subsection{Outline of possible counter-examples to the natural conjecture}
\label{sec:counter}
Take $r_i \to \infty$ very fast and $\delta_i \to 0$ very fast.
Draw a line segment between the  Poisson point pairs which are at some distance in $\cup_i [r_i,r_i+\delta_i]$.
One can arrange that the density of intersections of these lines is arbitrarily small.  
Break the (rare) circuits.  Then assign random arrival times  and use a ``Poisson rain" construction as in the section above.
In this way it might be possible to construct an invariant tree-network such that $\rho(r) < \infty$ for $r \in  \cup_i [r_i,r_i+\delta_i]$.

\subsection{General spatial networks}
\label{sec:rem}
For general (i.e. non-tree) invariant networks over Poisson points, the quantity
\[ \rho(r) := \mbox{ mean route length between two Poisson points at distance } r  \]
at (\ref{def:rho}) is a natural object of study.  
From  \cite{me-arx-1,hirsch} we know that
under very weak assumptions 
(which roughly correspond to ``not a tree"), not only is 
$\rho(r) < \infty$ for all $r$, but also (by subadditivity, heuristically) there exists the limit
\[ \lim_{r \to \infty} r^{-1}  \rho(r) := \rho^* < \infty .
\]
That is, average route lengths are asymptotically linear in straight-line distance. 
But quantitative analytic study of $\rho(r)$ or $\rho^*$ seems very difficult even in simple-to-describe network models.

For reasons explained  in \cite{shun}, it is not always wise to use $\rho^*$ 
as a summary statistic for efficiency at providing short routes. 
Instead, in \cite{shun} we recommend the statistic $\sup_r r^{-1}  \rho(r)$ to ensure that the network
provides short routes {\em on all scales}.  
This line of thought also motivates study of exactly self-similar networks (so $r^{-1}  \rho(r)$ is constant) on the
continuum plane \cite{ganesan}.

\paragraph{Acknowledgements.}
I thank Yuval Peres and Russ Lyons for the references to \cite{peres,lyons}, and Geoffrey Grimmett for comments on the contour method and for catching an error in an early draft.


\end{document}